\documentclass[sn-mathphys-num]{sn-jnl}

\usepackage{amsmath, amsthm, amsfonts}

\allowdisplaybreaks[3] 

\usepackage{hyperref}
\usepackage{color}
\usepackage{orcidlink}

\allowdisplaybreaks[3]

\theoremstyle{thmstyleone}%
\newtheorem{theorem}{Theorem}
\theoremstyle{thmstyletwo}%

\theoremstyle{thmstylethree}%

\newtheorem*{Thorne}{Thorne Theorem}
\newtheorem*{Sheffer}{Sheffer Theorem}
\newtheorem{corollary}[theorem]{Corollary}

\numberwithin{equation}{section} 

\newcommand{\Ea}{E_{\nu}}
\newcommand{\La}{\Lambda_{\nu}}

\newcommand{\binoma}[2]{\binom{#1}{#2}_{\!\nu}}

\newcommand{\I}{\mathcal{I}}

\newcommand{\B}{\mathfrak{B}}
\newcommand{\E}{\mathcal{E}}

\newcommand{\cc}{\gamma}

\newcommand{\sgn}{\mathrm{sgn}}

\newcommand{\LL}{\mathrm{L}}

\title[Two characterizations of Sheffer-Dunkl sequences]{Two characterizations of Sheffer-Dunkl sequences}

\author[1]{\fnm{Alejandro} \sur{Gil Asensi \orcidlink{0009-0007-9667-0744}}}\email{algila@unirioja.es}

\author*[1]{\fnm{Judit} \sur{M\'inguez Ceniceros\orcidlink{0000-0002-9840-5772}}}\email{judit.minguez@unirioja.es}
\equalcont{These authors contributed equally to this work.}

\affil*[1]{\orgdiv{Departamento de Matem\'aticas y Computaci\'on}, \orgname{Universidad de La Rioja}, \orgaddress{\city{Logro\~no}, \postcode{26006},  \country{Spain}}}

\begin{document}


\abstract{ Sheffer polynomials can be characterized using different Stieltjes integrals. These families of polynomials have been recently extended to the Dunkl context. In this way some classical operators as the derivative operator or the difference operator are replaced as analogous operators in the Dunkl universe. In this paper we establish two Stieltjes integrals that help us to characterize the Sheffer-Dunkl polynomials. }


\maketitle

\section{Introduction}

Let $g(t)$ and $f(t)$ be some formal power series given by
\begin{equation}\label{eq:functions}
g(t)=\sum_{n=0}^{\infty}\frac{b_n}{n!}t^n,\, b_0\ne 0,\quad f(t)=\sum_{n=1}^{\infty}\frac{a_n}{n!}t^n,\, a_1\ne 0.
\end{equation}
Let $\LL_f\colon \mathcal{P}\to \mathcal{P}$ (where $\mathcal{P}$ is the space of polynomials) be the linear operator given by
\begin{equation}\label{eq:operador}
\LL_fp(x):=\sum_{n=0}^{\infty}\frac{a_n}{n!}\frac{d^n}{dx^n}p(x).
\end{equation}
A \textit{Sheffer sequence} $\{s_n(x)\}_{n=0}^{\infty}$ for the pair $(g(t),f(t))$, where  $g(t)$ and $f(t)$ are as in~\eqref{eq:functions}, is defined using any of the two equivalent definitions:
\begin{itemize} 
	\item by a Taylor generating expansion
\begin{equation}
	\label{eq:generating}
	\frac{1}{g(\overline{f}(t))} e^{x\overline{f}(t)} = \sum_{n=0}^{\infty} s_n(x) \frac{t^n}{n!},
\end{equation}
where $\overline{f}(t)$ denotes the compositional inverse of $f(t)$;
\item by the linear operator
\begin{equation}\label{eq:op-pol}
	\LL_f s_n(x)=ns_{n-1}(x),\quad n\ge 1.
\end{equation}
\end{itemize}
Sheffer sequences were studied in \cite{C-H, RomV1, RomV2, RomV3, Rom, RomRot} using Umbral Calculus. Parti\-cular cases of Sheffer polynomials are \textit{Appell polynomials}, which are obtained taking $f(t)=t$, (in this case \eqref{eq:op-pol} is reduced to $s_n'(x)=ns_{n-1}(x)$ and $\overline{f}(t)=t$ in \eqref{eq:generating}) and \textit{associated polynomials} for $f(t)$, which are obtained taking $g(t)=1$.

Other characterizations of this kind of polynomials can be found in \cite{Sheffer} and \cite{Thorne} using Stieltjes integrals. Thorne in \cite{Thorne} gives a characterization for Appell polynomials that can be easily extended to Sheffer polynomials as points out Sheffer in \cite{Sheffer}. In particular, he proves that:
\begin{Thorne}
	A sequence of polynomials $\{A_n(x)\}_{n=0}^{\infty}$ is an Appell sequence ($f(t)=t$) if and only if there exists a function $\alpha(x)$ of bounded variation on $(0,\infty)$ with the following properties:
	\begin{enumerate}
		\item [i)] the moment integrals
		\[
		\mu_n=\int_0^{\infty}x^n\,d\alpha(x)
		\]
		all exist and $\mu_0\ne 0$;
		\item [ii)] if $A_n^{(r)}(x)$ denotes the $r$-derivative of $A_n(x)$,
		\[
		\int_0^{\infty}A^{(r)}_n(x)\,d\alpha(x)=n!\,\delta_{n,r},
		\]
		where $\delta_{n,r}=1$ if $n=r$ and $\delta_{n,r}=0$ if $n\ne r$.
	\end{enumerate}
In this case, the function $g(t)$ in \eqref{eq:generating} is given by
\[
g(t)=\sum_{n=0}^{\infty}\mu_n\frac{t^n}{n!}=\int_0^{\infty}e^{tx}\,d\alpha(x).
\]
\end{Thorne}
On the other hand, Sheffer in \cite{Sheffer} provides a characterization for the Sheffer sequences by raising a different integral.

\begin{Sheffer}
	A sequence of polynomials $\{s_n(x)\}_{n=0}^{\infty}$ is a Sheffer sequence for the pair $(g(t),f(t))$ if and only if there exists a function $\beta(x)$ of bounded variation on $(0,\infty)$ with the following properties:
	\begin{enumerate}
		\item [i)] the moment integrals
		\[
		\omega_n=\int_0^{\infty}x^n\,d\beta(x)
		\]
		all exist and $\omega_0\ne 0$;
		\item [ii)] the polynomials $s_n(x)$ can be expressed as
		\begin{equation}\label{eq:Sheffer-Sheffer}
		s_n(x)=\int_0^{\infty}p_n(x+t)^n\,d\beta(t),
		\end{equation}
		where $p_n(x)$ are the associated polynomials for $f(t)$.
	\end{enumerate}
	In this case, the function $g(t)$ is given by
	\[
	\frac{1}{g(t)}=\int_0^{\infty}e^{xt}\,d\beta(x)=\sum_{n=0}^{\infty}\omega_n\frac{t^n}{n!}.
	\]
\end{Sheffer}
Note that if $f(t)=t$ we have a characterization for Appell polynomials. In this case the associated polynomials in \eqref{eq:Sheffer-Sheffer} are given by $p_n(x)=x^n$.

The main complication when applying these theorems is to find, for a given sequence of moments $\{c_n\}_{n=0}^{\infty}$, the corresponding function $w(x)$ such that
\[
c_n=\int_{-\infty}^{\infty}x^n\,dw(x).
\]
A moment problem is said to have a solution if $w(x)$ is a positive measure. In~\cite{Duran} the author provides a technique to find the functions $w(x)$.

In \cite{GMV}, Sheffer polynomials are extended to Dunkl context. In this case, the derivative operator in \eqref{eq:operador} is replaced by the Dunkl operator
\begin{equation*}
	\La f(x) = \frac{d}{dx}f(x) + \frac{2\nu+1}{2}
	\left(\frac{f(x)-f(-x)}{x}\right),
\end{equation*}
where $\nu > -1$ is a fixed parameter (see \cite{Du, Ros}). Observe that the case $\nu=-1/2$ recovers the classical case $\Lambda_{-1/2} = \frac{d}{dx}$. The sequence of $n!$ is replaced by 
\begin{equation}
	\label{eq:ccna}
	\cc_{n,\nu} =
	\begin{cases} 
		2^{2k}k!\,(\nu+1)_k, & \text{if $n=2k$},\\
		2^{2k+1}k!\,(\nu+1)_{k+1}, & \text{if $n=2k+1$},
	\end{cases}
\end{equation}
where $(a)_n$ denotes the Pochhammer symbol
\[
(a)_n = a(a+1)(a+2) \cdots (a+n-1) = \frac{\Gamma(a+n)}{\Gamma(a)}
\]
(with $a$ a non-negative integer); of course, $\cc_{n,-1/2} = n!$.

In order to define the Sheffer-Dunkl polynomials we take two formal power series in this way 
\begin{equation}\label{eq:functions-D}
	g(t)=\sum_{n=0}^{\infty}\frac{b_n}{\cc_{n,\nu}}t^n,\, b_0\ne 0,\quad f(t)=\sum_{n=1}^{\infty}\frac{a_n}{\cc_{n,\nu}}t^n,\, a_1\ne 0.
\end{equation}
Then a\textit{ Sheffer-Dunkl sequence $\{s_{n,\nu}(x)\}_{n=0}^{\infty}$ for the pair $(g(t),f(t))$} as in~\eqref{eq:functions-D} satisfies that
\begin{equation}\label{eq:oper-lineal-D}
\LL_fs_{n,\nu}(x):=\sum_{n=1}^{\infty}\frac{a_n}{\cc_{n,\nu}}\La^n s_{n,\nu}(x)=\frac{\cc_{n,\nu}}{\cc_{n-1,\nu}}s_{n-1,\nu}(x),
\end{equation}
where $\La^0$ is the identity operator and $\La^{n+1} = \La(\La^n)$. Or, equivalently, 
\begin{equation}\label{eq:gener-S-F}
\frac{1}{g(\overline{f}(t))}\Ea(x\overline{f}(t))=\sum_{n=0}^{\infty}s_{n,\nu}(x)\frac{t^n}{\cc_{n,\nu}},
\end{equation}
where $\overline{f}(t)$ denotes the compositional inverse of $f(t)$ and $\Ea(t)$ is called Dunkl exponential or Dunkl kernel, and it is defined by
\[
\Ea(t) = \sum_{n=0}^\infty \frac{t^n}{\cc_{n,\nu}}.
\]
This function plays the role of the exponential function in the classical case. 
It satisfies $\La(\Ea(\lambda t)) = \lambda \Ea(\lambda t)$ and it can also be written as
\[
\Ea(t) = \I_\nu(t) + \frac{1}{2(\nu+1)}G_{\nu}(t),
\]
where
\begin{equation*}
	\label{eq:def-I}
	\I_\nu(t) = 2^\nu \Gamma(\nu+1)\frac{J_\nu(it)}{(it)^\nu} 
	= \sum_{n=0}^\infty \frac{t^{2n}}{\cc_{2n,\nu}}
\end{equation*}
and
\begin{equation*}
	\label{eq:def-G}
	G_\nu(t) = t \I_{\nu+1}(t) = \sum_{n=0}^\infty \frac{t^{2n+1}}{\cc_{2n,\nu}} 
	= 2(\nu + 1) \sum_{n=0}^\infty \frac{t^{2n+1}}{\cc_{2n + 1,\nu}}.
\end{equation*}
Here, $J_\nu(t)$ is the Bessel function of order $\nu$ (and hence, $\I_\nu(t)$ is a small variation of the modified Bessel function of the second kind, $I_\nu(t)$), see~\cite{Wat} or~\cite{OlMa-NIST}. Therefore, $\I_\nu(t)$ and $G_\nu(t)$ play the role of the even part and the odd part of $\Ea(t)$, respectively. We remark that, when $\nu=-1/2$, we recover the classical case, i.e., $\Lambda_{-1/2} = d/dt$, $E_{-1/2}(x) = e^t$, $\mathcal{I}_{-1/2}(t)=\cosh t$ and $G_{-1/2}(t)=\sinh t$.


Finally, we are going to need the \textit{translation operator} in this context. The Dunkl translation is  defined as
\begin{equation}
	\label{eq:translation}
	\tau_y f(x) = \sum_{n=0}^\infty \La^{n}f(x) \frac{y^n}{\cc_{n,\nu}},
	\qquad \nu>-1.
\end{equation}
This is just a generalization of the classical translation, but considered as a Taylor expansion of a function $f$ around a fixed point $x$, that is,
\[
f(x+y) = \sum_{n=0}^\infty f^{(n)}(x) \frac{y^n}{n!}.
\]
The Dunkl translation has a property than resembles the Newton binomial $(x+y)^n = \sum_{k=0}^n \binom{n}{k} y^k x^{n-k}$, which is
\begin{equation*}
	\tau_y ((\cdot)^n)(x) = 
	\sum_{k=0}^n \binoma{n}{k} y^k x^{n-k},
\end{equation*}
where
\[
\binoma{n}{j} = \frac{\cc_{n,\nu}}{\cc_{j,\nu}\cc_{n-j,\nu}}.
\]
Of course, $\binoma{n}{j}$ becomes the ordinary binomial numbers in the case $\nu=-1/2$.
Some other properties of the Dunkl translation, including an integral expression that is more general than~\eqref{eq:translation}, can be found in \cite{Ros} and~\cite{ThX}.

In that setting, an \textit{Appell-Dunkl sequence $\{A_{n,\nu}(x)\}_{n=0}^{\infty}$ for $g(t)$} is
a sequence of polynomials that satisfies
\begin{equation*}
	\La A_{n,\nu}(x) =\frac{\cc_{n,\nu}}{\cc_{n-1,\nu}}A_{n-1,\nu}(x)
\end{equation*}
and its generating function is 
\[
\frac{1}{g(t)}\Ea(xt)=\sum_{n=0}^{\infty}A_{n,\nu}(x)\frac{t^n}{\cc_{n,\nu}}.
\]
Also, the \textit{associated Dunkl polynomials $p_{n,\nu}(x)$ for a function $f(t)$} are defined by means of the generating function
\begin{equation*}\label{eq:gen-aso}
\Ea(x\overline{f}(t))=\sum_{n=0}^{\infty}p_{n,\nu}(x)\frac{t^n}{\cc_{n,\nu}}.
\end{equation*}
Other important family of Sheffer-Dunkl polynomials are \textit{discrete Appell-Dunkl polynomials}. They have been introduced and studied in \cite{ExLaMiVa}. The generating function of this kind of polynomials $\{d_{n,\nu}(x)\}_{n=0}^{\infty}$ is given by
\[
\frac{1}{g(\overline{G_{\nu}}(t))}\Ea(x\overline{G_{\nu}}(t))=\sum_{n=0}^{\infty}d_{n,\nu}(x)\frac{t^n}{\cc_{n,\nu}}.
\]
The operator $\LL_f$ for the discrete Appell-Dunkl polynomials is
\begin{equation}\label{eq:oper-dis}
	\LL_{G_{\nu}}=(\nu+1)(\tau_1-\tau_{-1}).
\end{equation}
And it holds that
\[
\LL_{G_{\nu}}d_{n,\nu}(x)=\frac{\cc_{n,\nu}}{\cc_{n-1,\nu}}d_{n-1,\nu}.
\]
If $g(t)=1$ we obtain the Dunkl factorial  polynomials $\{f_{n,\nu}(x)\}_{n=0}^{\infty}$ generated in the following way
\begin{equation}\label{eq:fact-Dunkl}
\Ea(x\overline{G_{\nu}}(t))=\sum_{n=0}^{\infty}f_{n,\nu}(x)\frac{t^n}{\cc_{n,\nu}}.
\end{equation}

The paper is structured in this way. Section 2 is devoted to extend the Thorne Theorem to the Sheffer-Dunkl sequences of polynomials. In Section 3 we characterized these polynomials extending the Sheffer Theorem to the Dunkl context. Finally, in Section 4 we show several examples of Sheffer-Dunkl polynomials where we contruct the corresponding measures and obtain properties of these families, some of them, unkown.

\section{Thorne Theorem for the Sheffer-Dunkl polynomials}

In this Section we present the Thorne Theorem extended to Sheffer-Dunkl polyno\-mials. 

\begin{theorem}\label{teo:Thorne-S-D}
	Let $g(t)$, $f(t)$ be two formal series as in \eqref{eq:functions-D}. Then a sequence  $\{s_{n,\nu}(x)\}_{n=0}^{\infty}$ is the Sheffer-Dunkl sequence for the pair $(g(t),f(t))$ if and only if there exists a function $\alpha_{\nu}(x)$ of bounded variation on $(-\infty,\infty)$ with the following properties
	\begin{enumerate}
		\item [i)] there exist the moment integrals 
		\[
		\mu_{n,\nu}=\int_{-\infty}^{\infty}x^n\,d\alpha_{\nu}(x), \quad n=0,1,\ldots,
		\]
		and $\mu_{0,\nu}\ne 0$;
		\item [ii)] the polynomials $s_{n,\nu}(x)$ satisfy
		\begin{equation}\label{eq:thorne-S-D}
			\int_{-\infty}^{\infty}\LL_f^r s_{n,\nu}(x)\,d\alpha_{\nu}(x)=\cc_{n,\nu}\delta_{n,r},
		\end{equation}
	where $\LL^r_f$ is the linear operator defined in \eqref{eq:oper-lineal-D} applied $r$ times. 
	\end{enumerate}
\end{theorem}
\begin{proof}
	Let $s_{n,\nu}(x)$ be Sheffer-Dunkl polynomials. We define $\mu_{n,\nu}$ as
	\[
	g(t)=\sum_{n=0}^{\infty}\mu_{n,\nu}\frac{t^n}{\cc_{n,\nu}},
	\]
	and $\mu_{0,\nu}\ne 0$ from \eqref{eq:functions-D}. By \cite[Chapter 3, \S 14, Theorem 14]{Widder}, there exists $\alpha_{\nu}(x)$ of bounded variation on $({-\infty},\infty)$ such that
	\[
	\mu_{n,\nu} = \int_{-\infty}^{\infty} x^n\,d\alpha_{\nu}(x)
	\]
	exists for $n=0,1,\ldots$ Then,
	\begin{equation}\label{eq:Thorne-g-S-D}
		g(t)=\int_{-\infty}^{\infty}\Ea(xt)\,d\alpha_{\nu}(x).
	\end{equation}
	Applying $r$ times the operator $\LL_f$  to the right part of \eqref{eq:gener-S-F}, and taking into account \eqref{eq:oper-lineal-D} we have
	\[
	\sum_{n=0}^{\infty}\LL^r_fs_{n,\nu}(x)\frac{t^n}{\cc_{n,\nu}}=t^r\frac{1}{g(\overline{f}(t))}\Ea(x\overline{f}(t)).
	\]
	Applying now the integral operator, from \eqref{eq:Thorne-g-S-D} we obtain
	\[
	\sum_{n=0}^{\infty}\int_{-\infty}^{\infty}\LL^r_fs_{n,\nu}(x)\,d\alpha_{\nu}(x)\frac{t^n}{\cc_{n,\nu}}=t^r,
	\]
	and equating coefficients \eqref{eq:thorne-S-D} is proved.
	
	Now, we suppose that i) and ii) hold. We are going to construct a  polynomial $s_{n,\nu}(x)$ of degree $n$ given by
	\[
	s_{n,\nu}(x)=c_nx^n+c_{n-1}x^{n-1}+\cdots+c_1x+c_0,
	\]
	and such that satisfies ii). Then we obtain the following system of equations
	\begin{equation*}\label{eq:sist}
		\left\{ \begin{array}{lllll}
			c_n \frac{\gamma_n}{\gamma_0}\mu_{0,\nu}  & & & & = \gamma_n, \\
			c_n \frac{\gamma_n}{\gamma_{1}}\mu_{1,\nu} &+ c_{n-1} \frac{\gamma_{n-1}}{\gamma_0}\mu_{0,\nu} & & & = 0, \\
				\vdots & \vdots & \vdots& & \vdots\\
			c_n \frac{\gamma_{n}}{\gamma_{n-r}}\mu_{n-r,\nu}  &+ \cdots &+ c_r \frac{\gamma_r}{\gamma_0}\mu_{0,\nu} & & = 0, \\
			\vdots & \vdots & \vdots& \vdots& \vdots\\
		c_n \mu_{n,\nu} &+ \cdots &+ \cdots &+ c_0 \mu_{0,\nu} & = 0.
		\end{array}\right.
	\end{equation*}
This system has a unique solution if $\mu_{0,\nu}\ne 0$. Now we have to see that these polynomials $s_{n,\nu}(x)$ are Sheffer-Dunkl polynomials. From ii)
\[
\int_{-\infty}^{\infty} \LL_f^{r+1}s_{n+1,\nu}(x)\,d\alpha_{\nu}(x)=\cc_{n+1,\nu}\delta_{n+1,r+1}.
\]
Let $r_{n,\nu}(x)=\frac{\cc_{n,\nu}}{\cc_{n+1,\nu}}\LL_f s_{n+1,\nu}(x)$, then
\[
\int_{-\infty}^{\infty} \LL_f^r r_{n,\nu}(x)\,d\alpha_{\nu}(x)=\frac{\cc_{n,\nu}}{\cc_{n+1,\nu}}\cc_{n+1,\nu}\delta_{n+1,r+1}=\cc_{n,\nu}\delta_{n,r}=\int_{-\infty}^{\infty} \LL_f^rs_{n,\nu}(x)\,d\alpha_{\nu}(x).
\]
Then, as the polynomials $s_{n,\nu}(x)$ are unique, we have that $r_{n,\nu}(x)=s_{n,\nu}(x)$ and 
\[
s_{n,\nu}(x)=\frac{\cc_{n,\nu}}{\cc_{n+1,\nu}}\LL_f s_{n+1,\nu}(x).
\]	
\end{proof}

\begin{corollary}
	\label{cor:g(t)-Thorne}
	Let $\{s_{n,\nu}(x)\}_{n=0}^{\infty}$ be the Sheffer-Dunkl sequence for the pair $(g(t),f(t))$. Then, it holds
	\begin{equation}\label{eq:g(t)-Thorne}
		g(t)=\int_{-\infty}^{\infty}\Ea(xt)\,d\alpha_{\nu}(x)=\sum_{n=0}^{\infty}\mu_{n,\nu}\frac{t^n}{\cc_{n,\nu}}.
	\end{equation}
\end{corollary}
\begin{proof}
	Applying the operator $\int_{-\infty}^\infty \LL^r(\cdot)\,d\alpha_{\nu}(x)$ to \eqref{eq:gener-S-F}, we have
	\[
	t^r\frac{1}{g(\overline{f}(t))}\int_{-\infty}^{\infty}\Ea(x\overline{f}(t))\,d\alpha_{\nu}(x)=t^r,
	\]
	and this implies
	\[
	g(\overline{f}(t))= \int_{-\infty}^{\infty} \Ea(x\overline{f}(t))\,d\alpha_{\nu}(x).
	\]
\end{proof}

\section{Sheffer Theorem for the Sheffer-Dunkl polynomials}

Now, we are going to generalize the Sheffer Theorem to the Dunkl case.

\begin{theorem}\label{thm:Sheffer-Dunkl} Let $f(t)$ and $g(t)$ be formal power series as in \eqref{eq:functions-D}. A sequence of polynomials $\{s_{n,\nu}(x)\}_{n=0}^{\infty}$ is the Sheffer-Dunkl sequence for $(g(t),f(t))$ if and only of there exists a function $\beta_{\nu}(x)$ of bounded variation on $({-\infty},\infty)$ such that
\begin{enumerate}
\item [i)] there exist the moment integrals
\[
\omega_{n,\nu}=\int_{-\infty}^{\infty} x^n\,d\beta_{\nu}(x),\quad n=0,1,\ldots,
\]
and $\omega_{0,\nu}\ne 0$.
\item [ii)] the polynomials $s_{n,\nu}(x)$ can be expressed as
\begin{equation}\label{eq:S-F-integral}
s_{n,\nu}(x)=\int_{-\infty}^{\infty}\tau_t(p_{n,\nu})(x)\,d\beta_{\nu}(t),
\end{equation}
where $\{p_{n,\nu}(x)\}_{n=0}^{\infty}$ is the sequence of associated Dunkl polynomials for the function $f(t)$.
\end{enumerate}
\end{theorem}
\begin{proof}

If the condition i) holds it is clear that $\{s_{n,\nu}(x)\}_{n=0}^{\infty}$ given by \eqref{eq:S-F-integral} is a sequence of Sheffer-Dunkl polynomials for $(g(t),f(t))$ because
\begin{multline*}
\LL_f s_{n,\nu}(x)=\int_{-\infty}^{\infty}\LL_f\circ\tau_t(p_{n,\nu})(x)\,d\beta_{\nu}(t)=\int_{-\infty}^{\infty}\tau_t\circ\LL_f(p_{n,\nu})(x)\,d\beta_{\nu}(t)\\=\int_{-\infty}^{\infty}\frac{\cc_{n,\nu}}{\cc_{n-1,\nu}}\tau_t(p_{n-1,\nu})(x)\,d\beta_{\nu}(t)=\frac{\cc_{n,\nu}}{\cc_{n-1,\nu}}s_{n-1,\nu}(x).
\end{multline*}
 In the second equality, we have used that the linear operator $\LL_f$ and the translation operator $\tau_t$ commute (see \cite{GMV}).
 
 Now, let $\{s_{n,\nu}(x)\}_{n=0}^{\infty}$ be the Sheffer-Dunkl sequence for $(g(t),f(t))$. We write the power series of the function
 \[
 \frac{1}{g(t)}=\sum_{n=0}^{\infty} b_{n}\frac{t^n}{\cc_{n,\nu}}.
 \]
 We define $\omega_{n,\nu}=b_n$ and let $\beta_{\nu}(t)$ be a function of bounded variation on $({-\infty},\infty)$, guaranteed by the Boas theorem (\cite[Chapter 3, \S 14, Theorem 14]{Widder}), whose moments are $\{\omega_{n,\nu}\}_{n=0}^{\infty}$. With this function, $\beta_{\nu}(t)$, we take the following Sheffer-Dunkl sequence
 \begin{equation}\label{eq:aux}
 Q_n(t)=\int_{-\infty}^{\infty}\tau_t(p_{n,\nu})(x)\,d\beta_{\nu}(t),
 \end{equation}
 for a pair $(g_1(t),f(t))$. We have to see that $g_1(t)=g(t)$. From \eqref{eq:gener-S-F} and \eqref{eq:aux}  we have
 \[
 \frac{1}{g_1(\overline{f}(t))}\Ea(x\overline{f}(t))=\sum_{n=0}^{\infty}\frac{1}{\cc_{n,\nu}}\int_{-\infty}^{\infty}\tau_u(p_{n,\nu})(x)\,d\beta_{\nu}(u)t^n.
 \]
 Using the binomial property for the associated Dunkl polynomials proved in \cite{GMV}
 \[
 \tau_u(p_{n,\nu})(x)=\sum_{k=0}^n\binoma{n}{k}p_{k,\nu}(x)p_{n-k,\nu}(u),
 \]
 we have 
 \begin{align*}
 \frac{1}{g_1(\overline{f}(t))}\Ea(x\overline{f}(t))&=\sum_{k=0}^{\infty}\sum_{k=0}^n\frac{p_{k,\nu}(x)t^k}{\cc_{k,\nu}}\int_{-\infty}^{\infty}\frac{p_{n-k,\nu}(u)t^{n-k}}{\cc_{n-k,\nu}}\,d\beta_{\nu}(u)\\&=\Ea(x\overline{f}(t))\int_{-\infty}^{\infty}\Ea(u\overline{f}(t))\,d\beta_{\nu}(u)\\&=\sum_{n=0}^{\infty} \int_{-\infty}^{+\infty}u^n\,d\beta_{\nu}(u)\frac{\overline{f}(t)^n}{\cc_{n,\nu}}\\&= \frac{1}{g(\overline{f}(t))}\Ea(x\overline{f}(t)).
 \end{align*}
 Then, $g(t)=g_1(t)$ and $s_{n,\nu}(x)=Q_n(x)$.
 
\end{proof}
From the proof of Theorem \ref{thm:Sheffer-Dunkl} we deduce that:
\begin{corollary}
	\label{cor:g(t)-Sheffer}
	Let $\{s_{n,\nu}(x)\}_{n=0}^{\infty}$ be the Sheffer-Dunkl sequence for a pair $(g(t),f(t))$. Then, it holds
	\begin{equation}\label{eq:g(t)-Sheffer}
		g(t)=\left(\int_{-\infty}^{\infty}\Ea(xt)\,d\beta_{\nu}(x)\right)^{-1}=\left(\sum_{n=0}^{\infty}\omega_{n,\nu}\frac{t^n}{\cc_{n,\nu}}\right)^{-1}.
	\end{equation}
\end{corollary}

\section{Examples }
In this Section we are going to show some examples of moment problems for different families of Sheffer-Dunkl polynomials. 
\subsection{Truncated polynomials}

In \cite{JMC} the following family of Appell-Dunkl polynomials $\{A_{n,\nu}(x)\}_{n=0}^{\infty}$ were studied
\[
\frac{\Ea(xt)}{1-t}=\sum_{n=0}^{\infty}A_{n,\nu}(x)\frac{t^n}{\cc_{n,\nu}}.
\]
In this case,
\[
\frac{1}{g(t)}=\frac{1}{1-t}=\sum_{n=0}^{\infty}\frac{\cc_{n,\nu}}{\cc_{n,\nu}}t^n,
\]
and
\[
g(t)=1-t.
\]
We start giving the function $\alpha_{\nu}(x)$ corresponding to the Theorem~\ref{teo:Thorne-S-D}. From \eqref{eq:g(t)-Thorne}, we are looking for a function whose moments $\{\mu_{n,\nu}\}_{n=0}^{\infty}$ are
\[
\mu_{0,\nu}=1,\quad \mu_{1,\nu}=-\cc_{1,\nu},\quad \mu_{n,\nu}=0,\, n\ge 2.
\]
We are going to follow \cite{Duran} to find $\alpha_{\nu}(x)$. First of all, we construct the auxiliar function
\begin{equation}\label{eq:fun-aux}
F(t)=\sum_{n=0}^{\infty}\frac{i^n}{2\pi n!}\mu_{n,\nu} t^n,
\end{equation} 
that in this case it is given by
\[
F(t)=\frac{1}{2\pi}(1-i\cc_{1,\nu}t).
\]
Now, we make the Fourier transform of $f(t)$ and we obtain that
\begin{equation}\label{eq:medida-T-S}
\alpha_{\nu}(x)=\delta_0(x)+\cc_{1,\nu}\delta_0'(x),
\end{equation}
where $\delta_0$ is the Dirac delta. We are going to verify that it is the right measure. Our polynomials $A_{n,\nu}(x)$ are
\[
A_{n,\nu}(x)=\cc_{n,\nu}\sum_{k=0}^n\frac{x^k}{\cc_{k,\nu}},
\]
and the corresponding operator $\LL_f$ is $\La$ because they are Appell-Dunkl polynomials. It is easy to check that
\[
\int_{-\infty}^{\infty}\La^r A_{n,\nu}(x)\,d\alpha_{\nu}(x)=\frac{\cc_{n,\nu}}{\cc_{n-r,\nu}}A_{n-r,\nu}(0)-\frac{\cc_{n,\nu}}{\cc_{n-r,\nu}}\cc_{1,\nu}A'_{n-r,\nu}(0)=\cc_{n,\nu}\delta_{n,r}.
\]

On the other hand, we can try to find the function that holds Theorem~\ref{thm:Sheffer-Dunkl} for these polynomials. In this case, from \eqref{eq:g(t)-Sheffer}, we are looking for a function of bounded variation $\beta_{\nu}(x)$ whose moments are
\[
\omega_{n,\nu}=\cc_{n,\nu},\quad n=0,1,\ldots.
\]
In \cite{DGV}, it is proved that this function is a positive measure  if and only if $-1<\nu\le -1/2$ and the measure is given by
\begin{equation}\label{eq:medida-gamma}
\beta_{\nu}(x)=\frac{|x|^{\nu+1}(K_{\nu}(|x|)+\sgn K_{\nu+1}(|x|))}{2^{\nu+1}\Gamma(\nu+1)},
\end{equation}
where $K_{\mu}$ is the modified Bessel function of the second kind. So,
\[
A_{n,\nu}(x)=\int_{-\infty}^{\infty}\tau_t(\cdot)^n(x)\,d\beta_{\nu}(t).
\]

\subsection{Discrete truncated Appell-Dunkl polynomials}

Now, we are going to study the moment problems for the discrete case for the truncated Appell-Dunkl polynomials. Let $\{a_{n,\nu}(x)\}_{n=0}^{\infty}$ be the discrete truncated Appell-Dunkl polynomials whose genera\-ting function is
\[
\frac{1}{1-\overline{G_{\nu}}(t)}\Ea(x\overline{G_{\nu}}(t))=\sum_{n=0}^{\infty}a_{n,\nu}(x)\frac{t^n}{\cc_{n,\nu}}.
\]
Then, as we have proved in the previous example, the function $\alpha_{\nu}(x)$ for Theorem~\ref{teo:Thorne-S-D} is~\eqref{eq:medida-T-S} and it holds
\[
\int_{-\infty}^{\infty}\LL^r_{G_{\nu}}a_{n,\nu}(x)\,d\alpha_{\nu}(x)=\frac{\cc_{n,\nu}}{\cc_{n-r,\nu}}(a_{n-r,\nu}(0)-\cc_{1,\nu}a'_{n-r}(0))=\cc_{n,\nu}\delta_{n,r}.
\]
On the other hand, \eqref{eq:medida-gamma} is the measure corresponding to Theorem~\ref{thm:Sheffer-Dunkl} and it holds
\[
a_{n,\nu}(x)=\int_{-\infty}^{\infty}\tau_t(f_{n,\nu})(x)\,d\beta_{\nu}(t),
\]
where $\{f_{n,\nu}(x)\}_{n=0}^{\infty}$ are the Dunkl factorial polynomials \eqref{eq:fact-Dunkl}. 

\subsection{Other Appell-Dunkl polynomials}

Let $\{A_{n,\nu}(x)\}_{n=0}^{\infty}$ be the sequence of Appell-Dunkl polynomials defined with the ge\-nerating function
\[
\frac{\Ea(xt)}{1-t^2}=\sum_{n=0}^{\infty}A_{n,\nu}(x)\frac{t^n}{\cc_{n,\nu}}.
\]
These polynomials are expressed as
\[
A_{n,\nu}(x)=\cc_{n,\nu}\sum_{k=0}^{[n/2]}\frac{x^{n-2k}}{\cc_{n-2k,\nu}}.
\]
Then, for finding the function of Theorem~\ref{teo:Thorne-S-D} we have to study the function
\[
g(t)=1-t^2.
\]
From \eqref{eq:g(t)-Thorne}, the moments are 
\[
\mu_{0,\nu}=1,\quad \mu_{1,\nu}=0,\quad \mu_{2,\nu}=-\cc_{2,\nu},\quad \mu_{n,\nu}=0,\quad n\ge 3.
\]
The auxliar function \eqref{eq:fun-aux} for them is
\[
F(t)=\frac{1}{2\pi}\left(1+\frac{\cc_{2,\nu}}{2}t^2\right),
\]
and, following the technique of~\cite{Duran}, the function will be
\[
\alpha_{\nu}(x)=\delta_0(x)-\frac{\cc_{2,\nu}}{2}\delta_0''(x).
\]
It is easy to check that
\[
\int_{-\infty}^{\infty} \La A_{n,\nu}(x)\,d\alpha_{\nu}(x)=\frac{\cc_{n,\nu}}{\cc_{n-r,\nu}}(A_{n-r,\nu}(0)-\frac{\cc_{2,\nu}}{2}A_{n,\nu}''(0))=\cc_{n,\nu}\delta_{n,r}.
\]

In order to obtain the function of the Theorem~\ref{thm:Sheffer-Dunkl}, we take the function
\[
\frac{1}{g(t)}=\frac{1}{1-t^2}=\sum_{n=0}^{\infty}\cc_{2n,\nu}\frac{t^{2n}}{\cc_{2n,\nu}},
\]
and the moments will be
\[
\omega_{n,\nu}=\left\{\begin{array}{ll}
	\cc_{n,\nu}, & n=2k,\\
	0, & n=2k+1.
	\end{array}\right.
\]
This moment problem was also studied in~\cite{DGV} and they obtained the function
\[
\beta_{\nu}(x)=\frac{|x|^{\nu+1}K_{\nu}(|x|)}{2^{\nu+1}\Gamma(\nu+1)},
\]
that it is a positive measure if $\nu>-1$. 

\subsection{Bernoulli-Dunkl polynomials}

Bernoulli-Dunkl polynomials were introduced in \cite{CDPV} to sum series of zeros of Bessel functions. Moreover, they have been also studied in \cite{CMV} and \cite{MV}. Bernoulli-Dunkl polynomials $\B_{n,\nu}(x)$ are defined by means of the generating function
\[
\frac{\Ea(xt)}{\I_{\nu+1}(t)}=\sum_{n=0}^{\infty}\B_{n,\nu}(x)\frac{t^n}{\cc_{n,\nu}}.
\]
In this case,
\[
g(t)=\I_{\nu+1}(t) = \sum_{n=0}^\infty \frac{t^{2n}}{\cc_{2n,\nu+1}} = \sum_{k=0}^\infty \mu_{k,\nu} \frac{t^{k}}{\cc_{k,\nu}},
\]
and the moments for Theorem~\ref{teo:Thorne-S-D} are
\[
\mu_{n,\nu}=\begin{cases}
	\frac{\cc_{2k,\nu}}{\cc_{2k,\nu+1}},& n=2k,\\
	0,& n=2k+1.
\end{cases}
\]
From \eqref{eq:ccna} 
\[
\frac{\cc_{2k,\nu}}{\cc_{2k,\nu+1}}=\frac{(\nu+1)_k}{(\nu+2)_k}=\frac{\nu+1}{\nu+k+1}.
\]
%
The corresponding function \eqref{eq:fun-aux} is
\[
F(t)=\sum_{k=0}^{\infty}\frac{i^{2k}}{2\pi(2k)!}\frac{(\nu+1)_k}{(\nu+2)_k}t^{2k}= \frac{1}{2\pi}\,{}_1F_2(\nu+1,\nu+2,1/2;-t^2/4).
\]
We calculate its Fourier transform (which is \cite[p.61, eq.(5)]{Erdelyi}) and we obtain the measure
%
%
\begin{equation}\label{eq:measure-B-D}
\alpha_{\nu}(x) 
= \begin{cases}
	0, & x<-1,\\
(\nu+1)|x|^{2\nu+1}, & -1< x <1, \\
0 & x > 1,
\end{cases}
\end{equation}
valid for $\nu > -1$.
So,
\[
\int_{-\infty}^{\infty}\frac{\cc_{n,\nu}}{\cc_{n-r,\nu}}\B_{n-r,\nu}(x)\,d\alpha_{\nu}(x)=\cc_{n,\nu}\delta_{n,r}.
\]
Using that $\B_{2n,\nu}(x)$ are even polynomials and $\B_{2n+1,\nu}(x)$ are odd polynomials, it holds that
\[
  \int_{0}^1 \B_{2n,\nu}(x)x^{2\nu+1}\,dx=\begin{cases}
  	0, & n>0,\\
  	1, & n=0.
  \end{cases}
\]
To study the function for Theorem~\ref{thm:Sheffer-Dunkl} we would need to study the problem moment for
\[
\omega_{n,\nu}=\B_{n,\nu}(0),
\]
that we do not know to solve it.

\subsection{Bernoulli-Dunkl polynomials of the second kind}

Bernoulli-Dunkl polynomials of the second kind are the discrete case corresponding to Bernoulli-Dunkl polynomials. They are defined and studied in \cite{ExLaMiVa}. A sequence of polynomials $\{b_{n,\nu}(x)\}_{n=0}^{\infty}$ is a sequence of Bernoulli-Dunkl of the second kind if they satisfy
\[
\frac{t}{\overline{G_\nu}(t)}\Ea(x\overline{G_{\nu}}(t))=\sum_{n=0}^{\infty}b_{n,\nu}(x)\frac{t^n}{\cc_{n,\nu}}.
\]
The operator $\LL_{G_{\nu}}$ is \eqref{eq:oper-dis} and the measure $\alpha_{\nu}(x)$ of Theorem~\ref{teo:Thorne-S-D} is the same that for Bernoulli-Dunkl polynomials \eqref{eq:measure-B-D}.

\subsection{Euler-Dunkl polynomials}

Euler-Dunkl polynomials, $\{\E_{n,\nu}(x)\}_{n=0}^\infty$, are studied in \cite{CMV,DPV,MV}. They are defined by means of the generating function
\[
\frac{\Ea(xt)}{\I_{\nu}(t)} 
= \sum_{n=0}^{\infty}\frac{\E_{n,\nu}(x)}{\gamma_{n,\nu}}t^n.
\]
In this case, we are going to give the measure corresponding to Theorem~\ref{teo:Thorne-S-D}. The power series of the function $\I_{\nu}(t)$ is 
\[
\I_{\nu}(t)=\sum_{k=0}^{\infty}\frac{t^{2k}}{\cc_{2k,\nu}}.
\]
So, we are looking for a function whose moments are
\[
\mu_{n,\nu}=\left\{\begin{array}{ll}
	1,& n=2k,\\
	0,& n=2k+1.
	\end{array}\right. 
\] 
With these moments our function \eqref{eq:fun-aux} is
\[
F(t)=\sum_{k=0}^{\infty}\frac{i^{2k}}{2\pi(2k)!}t^{2k}=\frac{1}{2\pi}\cos t.
\]
From this function, following \cite{Duran} we obtain that
\begin{equation}\label{eq:medida-E-D}
\alpha_{\nu}(x)=\frac{1}{2}(\delta_{-1}(x)+\delta_1(x)).
\end{equation}
Applying the operator $\int_{-\infty}^{\infty}d\alpha_{\nu}(x)$ to Euler-Dunkl polynomials, we obtain
\[
\int_{-\infty}^{\infty}\La^r\E_{n,\nu}(x)\,d\nu(x)=\frac{\cc_{n,\nu}}{\cc_{n-r,\nu}}\frac{\E_{n-r,\nu}(-1)+\E_{n-r,\nu}(1)}{2}=\cc_{n,\nu}\delta_{n,r}.
\]
We knew that it is true because as it was proven in~\cite{CMV} that $\E_{n,\nu}(1)=0=\E_{n,\nu}(-1)$ for $n\ge 1$.

To study the function for Theorem~\ref{thm:Sheffer-Dunkl} we would need to solve the problem moment for
\[
\omega_{n,\nu}=\E_{n,\nu}(0),
\]
that we do not know to solve it.

\subsection{Boole-Dunkl polynomials}

Boole-Dunkl polynomials are defined and studied in~\cite{GLMV} and they are the correspon\-ding discrete polynomials to Euler-Dunkl polynomials. A sequence of polynomials $\{e_{n,\nu}(x)\}_{n=0}^{\infty}$ is a sequence of Boole-Dunkl polynomials if its generating function is given by
\[
\frac{\Ea(x\overline{G_{\nu}}(t))}{\I_{\nu}(\overline{G{\nu}}(t))}=\sum_{n=0}^{\infty}e_{n,\nu}(x)\frac{t^n}{\cc_{n,\nu}}.
\]
The operator $\LL_{G_{\nu}}$ for the discrete Appell-Dunkl polynomials is \eqref{eq:oper-dis}.
In this case, the measure corresponding to Theorem~\ref{teo:Thorne-S-D} is the same that for Euler-Dunkl polynomials \eqref{eq:medida-E-D} and it holds
\[
\int_{-\infty}^{\infty}\LL_{G_{\nu}}^r e_{n,\nu}(x)\,d\alpha_{\nu}(x)=\frac{\cc_{n,\nu}}{\cc_{n-r,\nu}}\frac{e_{n-r,\nu}(-1)+e_{n-r,\nu}(1)}{2}=\cc_{n,\nu}\delta_{n,r}.
\]
This can be also deduce because as it was proven in~\cite{GLMV}, 
\[
e_{n,\nu}(x)=\sum_{l=0}^n\frac{l}{n}\binoma{n}{l}\E_{l,\nu}(x)\B_{n-l,\nu}(0),\, n\ge 1,
\]
where $\B_{n,\nu}(x)$ are the Bernoulli-Dunkl polynomials. So, $e_{n,\nu}(1)=0=e_{n,\nu}(-1)$, $n\ge 1$.

\section*{Declarations}

\textbf{Ethical approval:} Not applicable.\\

\noindent\textbf{Conflict of interest:} The authors declare no competing interests.\\

\noindent\textbf{Funding:}  Partially supported by grant PID2021-124332NB-C22 
(Mi\-nis\-te\-rio de Cien\-cia e Inno\-va\-ci\'on-Agen\-cia Esta\-tal de Inves\-ti\-ga\-ci\'on, Spain).



\end{document}